\documentclass{article}
\usepackage{pgfplots}
\pgfplotsset{compat=1.9}

\usepackage[T2A]{fontenc}
\usepackage[cp1251]{inputenc}
\usepackage[english]{babel}
\usepackage[tbtags]{amsmath}
\usepackage{amsthm}
\usepackage{amsfonts,amssymb}

\usepackage{graphicx}

\newtheorem{Oldtheorem}{Theorem}

\newtheorem{Example}{Theorem}
\newtheorem{Conjecture}{Conjecture}

\theoremstyle{definition}
\newtheorem{Remark}{Remark}

\begin{document}


\date{}

\author{I.V.~Limonova\thanks{Lomonosov Moscow State University, Moscow Center for Fundamental and Applied Mathematics.}}


\title{On exact discretization of the $L_2$--norm with a negative weight.}

\maketitle

\begin{abstract}
	For a subspace $X$ of functions from $L_2$ we consider the minimal number $m$ of nodes necessary for the exact discretization of the $L_2$--norm of the functions in $X$. We construct a subspace such that for any exact discretization with $m$ nodes there is at least one  negative weight.
\end{abstract}


\bigskip

The paper is devoted to the problem posed in \cite{DPTT} on the exact discretization of the Marcinkiewich--type with weights.

Let $\Omega\subset \mathbb{R}^d$ be a compact subset, $\mu$ be a finite measure on $\Omega$, and $X_N$ be an $N$--dimensional subspace of the real space $L_2(\Omega, \mu)$ of individual functions.  From now on we assume that for any function $f\in X_N$ the value $f(x)$ is defined for any $x\in\Omega$. That is $X_N=\langle f_1,\dots, f_N \rangle$  (here and below the angle brackets denote the linear hull), where the functions $f_1,\dots, f_N\in L_2(\Omega, \mu)$ are linearly independent and defined everywhere on $\Omega$.

We say that a linear subspace $X_N$ admits {\bf the exact weighted Marcinkiewich-type discretization theorem} with parameters $m$ and $2$ if there exist a set of nodes $\{\xi^j\}_{j=1}^m\subset\Omega$ and a set of weights $\{\lambda_j\}_{j=1}^m$ such that for any $f\in X_N$ we have 
\begin{equation}\label{discr}
	\int\limits_{\Omega}f^2 d\mu=\sum\limits_{j=1}^m \lambda_{j}f^2(\xi^{j}).
\end{equation}
Here is a brief way to express the above property: $X_N\in\mathcal{M}^w(m,2,0)$. If a subspace $X_N$ admits equality \eqref{discr} with positive weights, we write $X_N\in\mathcal{M}_+^w(m,2,0)$.

It is known that $X_N\in \mathcal{M}^w(N(N+1)/2,2,0)$ (see \cite[Theorem $3.1$]{DPTT}). For a subspace $X_N\subset L_2(\Omega, \mu)$ we define
$$
m(X_N, w):=\min\{m: X_N\in\mathcal{M}^w(m,2,0)\}.
$$
The following conjecture was formulated in \cite{DPTT}: 
\begin{Conjecture}[{\cite[Open Problem $3$]{DPTT}}]\label{C1}
	Let $m=m(X_N, w)$ and  the nodes $\{\xi^{j}\}_{j=1}^m\subset\Omega$ and the weights $\{\lambda_{j}\}_{j=1}^m$ be such that \eqref{discr} holds for any $f\in X_N$. Then  $\lambda_{j}>0$, $j=1,\dots, m$.
\end{Conjecture}

In Example 1 below we show that Conjecture \ref{C1} is not true. Note that  \cite{DPTT} deals with the probability space $(\Omega,\mu)$. But we can always come to this case by replacing $\mu$ with $\mu/\mu{(\Omega)}$, $\lambda_{j}$ with $\lambda_{j}/\mu{(\Omega)}$, $j=1,\dots, m$ (equality \eqref{discr} will remain true).

Let $\Omega=[-1,1]$, $\mu$ be the standard Lebesgue measure. 
Let $a>0$, $A>B>0$ be such that
\begin{equation}\label{cond_a}
A^2> 2a^2+B^2.
\end{equation}
For example, \eqref{cond_a} holds with $A=3$, $B=2$, $a=1$.

We define the functions $f_1$ and $f_2$ as follows:
\begin{equation*}
    f_1(x) = 
	\begin{cases}
	0, &\text{for $-1\leq x< 0$},\\
	1, &\text{for $0\le x\le 1$},
    \end{cases}
    \qquad
f_2(x) = 
	\begin{cases}
	 a, &\text{for $-1\leq x< 0$},\\
	 A, &\text{for $0\le x<1/4$},\\
	-A, &\text{for $1/4\le x<1/2$},\\
	 B,& \text{for $1/2\le x<3/4$},\\
	-B,& \text{for $3/4\le x\le 1$}.\\
	\end{cases}
\end{equation*}
Then 
\begin{align}\label{int_values_ex1}
	\int\limits_{\Omega}f_1^2(x) d\mu=1, \quad \int\limits_{\Omega}f_2^2(x) d\mu=a^2+\frac{A^2}{2}+\frac{B^2}{2}.
\end{align}
Note that
\begin{equation}\label{f_orth}
\int\limits_{\Omega}f_1(x)f_2(x) d\mu=0.
\end{equation}

\begin{Example}[Example 1]\label{ex1}
The subspace $X_2=\langle f_1, f_2\rangle\subset L_2 (\Omega)$ has the following properties:
	\begin{itemize}
		\item[a)] $m(X_2, w)=3$,
		\item[b)] $X_2\in\mathcal{M}_+^w(3,2,0)$,
		\item[c)] for $\xi^1=-1/2$, $\xi^2=1/8$, $\xi^3=3/8$, $\lambda_1=(2a^2-A^2+B^2)/(2a^2)<0$, $\lambda_2=\lambda_3=1/2$, there holds \eqref{discr}  for all $f\in X_2$.
	\end{itemize} 
\end{Example}
\begin{Remark}\label{remark_f_g} 
It follows from {\it a)} and {\it c)}  that  Conjecture \ref{C1} is not true. 
\end{Remark}

\begin{proof}
Any function $h$ in $X_2$ has a representation $h=\alpha f_1+\beta f_2$, $\alpha,\beta\in\mathbb{R}$. From here and from \eqref{f_orth} we obtain that in order to have
	\begin{equation}\label{int_eq_k}
		\int\limits_{\Omega}f^2(x) d\mu=\sum\limits_{j=1}^k \lambda_jf^2(\xi^j),
	\end{equation}
for some $k\in\mathbb{N}$  for all functions $f\in X_2$ it is nessessary and sufficient that \eqref{int_eq_k} is satisfied for $f_1$ and $f_2$ and the following equality holds:
	\begin{equation}\label{discr_orth_0}
		\sum\limits_{j=1}^k \lambda_jf_1(\xi^j)f_2(\xi^j)=0.
	\end{equation}

It is clear that $X_2\notin\mathcal{M}^w(1,2,0)$, otherwise we can find $\xi^1\in\Omega$, $\lambda_1\in\mathbb{R}$ such that for $k=1$ equality \eqref{int_eq_k} holds for all $f\in X_2$, in particular, for $f_1$.  It follows from \eqref{int_values_ex1} that in this case $f_1(\xi^1)=1$, $\lambda_1=1$. Therefore from \eqref{discr_orth_0}  we obtain $f_2(\xi^1)=0$, which contradicts equality  \eqref{int_eq_k} for $f_2$.

Assume that $X_2\in\mathcal{M}^w(2,2,0)$, then \eqref{int_eq_k} holds with $k=2$ for some $\xi^1, \xi^2\in\Omega$, $\lambda_1, \lambda_2\in\mathbb{R}$. Let us consider all possible cases. 

If $f_1(\xi^1)=0$, then  we obtain from \eqref{discr_orth_0} that $\lambda_2f_1(\xi^2)f_2(\xi^2)=0$. If we have at the same time $f_1(\xi^2)=0$, then \eqref{int_eq_k} does not hold for $f_1$, so in this case $f_1(\xi^2)=1$. By construction, $f_2(\xi^2)\neq 0$, hence, $\lambda_2=0$ and $X\in\mathcal{M}^w(1,2,0)$, which is not true. 

Thus, $f_1(\xi^1)=1$. Similarly, $f_1(\xi^2)=1$. Then \eqref{int_eq_k} for $f_1$  takes the following form:
\begin{gather}	
\lambda_1+\lambda_2=1. \label{f_1}
\end{gather}
We see from \eqref{discr_orth_0} that 
\begin{equation}
	\lambda_1f_2(\xi^1)+\lambda_2f_2(\xi^2)=0. \label{orth}
\end{equation}

If $f_2(\xi^1)=f_2(\xi^2)$, then \eqref{orth} implies the equality $\lambda_1+\lambda_2=0$, which contradicts \eqref{f_1}. If $f_2(\xi^1)=-f_2(\xi^2)$, then we obtain from  \eqref{f_1} and \eqref{orth} that $\lambda_1=\lambda_2=1/2$. Then due to \eqref{cond_a} equality \eqref{int_eq_k} does not hold for $f_2$. The only case left is $|f_2(\xi^1)|\neq|f_2(\xi^2)|$. Without loss of generality we assume that $|f_2(\xi^1)|>|f_2(\xi^2)|$. 

In this case, \eqref{orth} implies $|\lambda_1|<|\lambda_2|$. We see from \eqref{int_eq_k} for $f_2$ and from relations \eqref{int_values_ex1} and \eqref{f_1} that
$$
\lambda_1A^2+(1-\lambda_1) B^2=a^2+\frac{B^2}{2}+\frac{A^2}{2}>\frac{A^2}{2}+\frac{B^2}{2}.
$$
Due to the monotonicity of the left--hand side of the above relation with respect to $\lambda_1$, we obtain $\lambda_1>1/2$. Then it follows from \eqref{f_1} that  $|\lambda_1|>|\lambda_2|$, so we came to a contradiction. Thus we established $m(X_2, w)>2$.

In order to prove c), it is sufficient to represent $f$ as a linear combination of $f_1$ and $f_2$ and put the values of the parameters in \eqref{discr}.

To establish b) we can take $\xi^1=-1/2$, $\xi^2=5/8$, $\xi^3=7/8$, $\lambda_1=(2a^2+A^2-B^2)/(2a^2)>0$, $\lambda_2=\lambda_3=1/2$ and do the same as in c). Now a) follows from the proved inequality $m(X_2, w)>2$ and b).

\end{proof}

The subspace $X_2$ from Example 1 has the following interesting property. There exist representations of the form \eqref{discr} (with $m=m(X_2, w)$) of both kinds: with all positive weights  and with one negative and two positive ones. It is natural to ask whether there exists a subspace $X$ such that  
$X\in \mathcal{M}^w(m,2,0) \backslash \mathcal{M}_+^w(m,2,0)$, where $m=m(X, w)$. We note that the functions in Example 1 are piecewise constant but could a similar property be true for a subspace consisting of continuous functions? A positive answer to both of these questions is given in Example 2 below.

Let us define functions $h_0, h_1,\dots, h_7$ which are continuos on $\Omega$. Let us divide them into four levels. The zero level consists of the function $h_0$, the first one --- of $h_1$, the second --- of $h_2$ and $h_3$, the third --- of $h_4$, $h_5$, $h_6$, and $h_7$. We introduce auxiliary continuos functions $g_{[a,b]}(x)$, $-1\leq a<b\leq 1$, on $\Omega$ as follows: denote $l=(b-a)/8$,
\begin{align*}
	g_{[a,b]}(x) = 
	\begin{cases}
	(x-a)/l, &\text{for $a\leq x<a+l$,}\\
	1, &\text{for $a+l\leq x<a+3l$,}\\
	(-x+a+4l)/l, &\text{for $a+3l\leq x< a+5l$,}\\\
	-1, &\text{for $a+5l\leq x<a+7l$,}\\
	(x-b)/l, &\text{for $a+7l\leq x\leq b$,}\\
	0, &\text{otherwise.}
	\end{cases}
\end{align*}

Define

\begin{minipage}{0.5\textwidth}
	\begin{align*}
		h_0(x)& = 
		\begin{cases}
		0, &\text{for $x\in[-1,0)$,}\\
		\sqrt{24x}, &\text{for $x\in[0,1/16)$,}\\
		\sqrt{-8x+2}, &\text{for $x\in[1/16,1/8)$,}\\
		1, &\text{for $x\in[1/8,1]$,}\\
		\end{cases}\\
		h_1(x)&=
		\begin{cases}
			-x/2, &\text{for $x\in[-1,0)$,}\\
			g_{[0,1]}(x), &\text{otherwise,}
		\end{cases}
	\end{align*}
  \end{minipage}
  \ 
  \begin{minipage}{0.45\textwidth}
	\begin{align*}
		h_2=&g_{[1/8, 1/4]}+\sqrt{5}g_{[1/4,3/8]},\\
	h_3=&g_{[5/8, 3/4]}+\sqrt{5}g_{[3/4,7/8]},\\
	h_4=&g_{[9/64,5/32]}+\sqrt{23}g_{[5/32,11/64]},\\
	h_5=&g_{[13/64,7/32]}+\sqrt{23}g_{[7/32,15/64]},\\
	h_6=&g_{[41/64,21/32]}+\sqrt{23}g_{[21/32,43/64]},\\
	h_7=&g_{[45/64,23/32]}+\sqrt{23}g_{[23/32,47/64]}.
	\end{align*}
  \end{minipage}

\bigskip
Note that functions of the same level can be transformed in one another by shifting the argument. By the support of a function we understand the set of points at which the function is not equal to zero. The supports of the functions $h_i$, $i=2,\dots,7$, belong to the segment $[1/8,1]$ on which the function $h_0$ equals $1$. We illustrate here the graphs of the functions $h_0$, $h_1$, $h_2$, $h_3$ for clarity.
\bigskip

\hskip 3pt
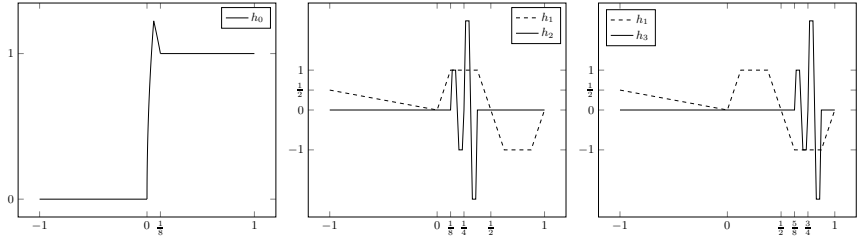
\begin{tikzpicture}[scale=0.5] 
	\begin{axis}[
		xticklabel style={/pgf/number format/frac, /pgf/number format/frac shift=2},
		xtick ={-1, 0, 1/8, 1},
		ytick ={0,1},
	]
		\addplot[color=black, domain=-1:0, samples=300]{0};
		\addplot[color=black, domain=0:{1/16}, samples=300]{sqrt(24*x)};
		\addplot[color=black, domain={1/16}:{1/8}, samples=300]{sqrt((-8)*x+2)};
		\addplot[color=black, domain={1/8}:{1}, samples=300]{1};
		\addlegendentry{$h_0$}
	\end{axis}
\end{tikzpicture}
\begin{tikzpicture}[scale=0.5] 
	\begin{axis}[
		xticklabel style={/pgf/number format/frac, /pgf/number format/frac shift=1},
		yticklabel style={/pgf/number format/frac, /pgf/number format/frac shift=1},
		xtick ={-1, 0, 1/8, 1/4, 1/2, 1},
		ytick ={-1,0,1/2,1},
	]
		\addplot[color=black, dashed] 
		coordinates {
		(-1,1/2)
		(0,0)
		(1/8, 1)
		(3/8,1)
		(5/8,-1)
		(7/8,-1)
		(1,0)
		};
		\addlegendentry{$h_1$}
		\addplot[color=black]
		coordinates {
		(-1,0)
		(1/8,0)
		(9/64, 1)
		(11/64,1)
		(13/64,-1)
		(15/64,-1)
		(4/16,0)
		(9/64+1/8, 2.236)
		(11/64+1/8, 2.236)
		(13/64+1/8,-2.236)
		(15/64+1/8, -2.236)
		(16/64+1/8, 0)
		(1,0)
		};
		\addlegendentry{$h_2$}
	\end{axis}
\end{tikzpicture}
\hskip 0pt
\begin{tikzpicture}[scale=0.5]
	\begin{axis}[
		xticklabel style={/pgf/number format/frac, /pgf/number format/frac shift=1},
		yticklabel style={/pgf/number format/frac, /pgf/number format/frac shift=1},
		xtick ={-1, 0, 1/2, 5/8, 3/4,1},
		ytick ={-1,0,1/2,1},
		legend pos=north west,
	]
		\addplot[color=black, dashed] 
		coordinates {
		(-1,1/2)
		(0,0)
		(1/8, 1)
		(3/8,1)
		(5/8,-1)
		(7/8,-1)
		(1,0)
		};
		\addlegendentry{$h_1$}
		\addplot[color=black]
		coordinates {
		(-1,0)
		(1/8+1/2,0)
		(9/64+1/2, 1)
		(11/64+1/2,1)
		(13/64+1/2,-1)
		(15/64+1/2,-1)
		(12/16,0)
		(9/64+5/8, 2.236)
		(11/64+5/8, 2.236)
	    (13/64+5/8,-2.236)
		(15/64+5/8, -2.236)
		(16/64+5/8, 0)
		(1,0)
		};
		\addlegendentry {$h_3$}
	\end{axis}
\end{tikzpicture}

We also present the graphs of the functions
$h_2$, $h_4$, and $h_5$ on the segment $[0,1/2]$. The functions $h_3$, $h_6$, and $h_7$ have the same structure on the segment $[1/2,1]$. 

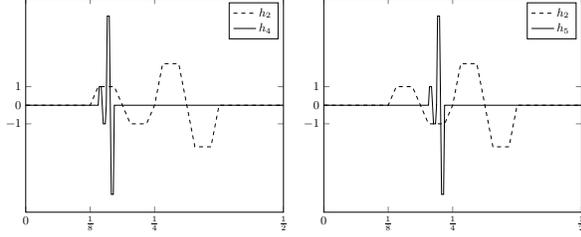
\begin{tikzpicture}[scale=0.5]
	\begin{axis}[
		xmin=0,
		xmax=1/2,
		xticklabel style={/pgf/number format/frac, /pgf/number format/frac shift=1},
		xtick ={0,1/8,1/4,1/2},
		ytick ={-1,0,1},
	]

		\addplot[color=black, dashed]
		coordinates {
		(-1,0)
		(1/8,0)
		(9/64, 1)
		(11/64,1)
		(13/64,-1)
		(15/64,-1)
		(4/16,0)
		(9/64+1/8, 2.236)
		(11/64+1/8, 2.236)
		(13/64+1/8,-2.236)
		(15/64+1/8, -2.236)
		(16/64+1/8, 0)
		(1,0)
		};
		\addlegendentry{$h_2$}

		\addplot[color=black]
		coordinates {
		(0,0)
		(9/64,0)
		(9/64+1/512,1)
		(9/64+3/512,1)
		(9/64+5/512,-1)
		(9/64+7/512,-1)
		(9/64+8/512,0)
		(5/32,0)
		(5/32+1/512, 4.7958)
		(5/32+3/512, 4.7958)
		(5/32+5/512,-4.7958)
		(5/32+7/512,-4.7958)
		(5/32+8/512,0)
		(1/2,0)
		};
		\addlegendentry{$h_4$}
	\end{axis}
\end{tikzpicture}
\begin{tikzpicture}[scale=0.5]
	\begin{axis}[
		xmin=0,
		xmax=1/2,
		xticklabel style={/pgf/number format/frac, /pgf/number format/frac shift=1},
		xtick ={0,1/8,1/4,1/2},
		ytick ={-1,0,1},
	]
		\addplot[color=black, dashed]
		coordinates {
		(-1,0)
		(1/8,0)
		(9/64, 1)
		(11/64,1)
		(13/64,-1)
		(15/64,-1)
		(4/16,0)
		(9/64+1/8, 2.236)
		(11/64+1/8, 2.236)
		(13/64+1/8,-2.236)
		(15/64+1/8, -2.236)
		(16/64+1/8, 0)
		(1,0)
		};
		\addlegendentry{$h_2$}
		\addplot[color=black]
		coordinates {
		(0,0)
		(13/64,0)
		(13/64+1/512,1)
		(13/64+3/512,1)
		(13/64+5/512,-1)
		(13/64+7/512,-1)
		(13/64+8/512,0)
		(7/32,0)
		(7/32+1/512, 4.7958)
		(7/32+3/512, 4.7958)
		(7/32+5/512,-4.7958)
		(7/32+7/512,-4.7958)
		(7/32+8/512,0)
		(1/2,0)
		};
		\addlegendentry{$h_5$}
	\end{axis}
\end{tikzpicture}

\begin{Example}[Example 2]\label{ex2} The subspace $X_8=\langle h_0, h_1, \dots, h_7\rangle\subset{L_2(\Omega)}$ has the following properties:
	\begin{itemize}
		\item [a)] $m(X_8, w)=9$,
		\item [b)] $X_8\in \mathcal{M}^w(9,2,0) \backslash \mathcal{M}_+^w(9,2,0)$.
	\end{itemize}
\end{Example}
\begin{proof}
For the functions $h_1,\dots, h_7$, the following two properties are fulfilled: the supports of the functions of the same level do not intersect and any function of a lower level is constant on the support of any function of a higher level. The orthogonality of the functions $h_0, h_1,\dots, h_7$ is easily derived from this. Therefore the following two conditions are nessessary and sufficient for \eqref{int_eq_k} to hold with some $k\in\mathbb{N}$ and for any $f\in X_8$:
\begin{equation}\label{for_h}
	\int\limits_{\Omega}h_i^2(x) d\mu=\sum\limits_{j=1}^k\lambda_jh_i^2(\xi^j),
\end{equation}
for $h_i$, $i=0,\dots, 7$,
\begin{equation}\label{for_pair}
	\sum\limits_{j=1}^k\lambda_jh_i(\xi^j)h_s(\xi^j)=0,
\end{equation}
for $h_i$ and $h_s$, $i\neq s$,  $i, s\in \{0,\dots, 7\}$.

Suppose that \eqref{int_eq_k} holds for some $k\in \mathbb{N}$ and  $\lambda_j\neq 0$, $j=1,\dots,k$.
Since the continuous functions $h_4, h_5, h_6, h_7$ are nonzero and  \eqref{for_h} holds for them, we see that there is at least one point from the set $\{\xi^j\}_{j=1}^k$ on the support of each of these functions. Note that if for some number $s\in\{4,5,6,7\}$ there is exactly one point on the support of the function $h_s$, then \eqref{for_pair} does not hold for $h_0$ and $h_s$. Therefore there are at least two points from the set $\{\xi^j\}_{j=1}^k$ on the support of each of the functions $h_4, h_5, h_6, h_7$. Since these supports do not intersect, we obtain  $m(X_8, w)\geq 8$.

Assume that $m(X_8, w)=8$. We established in the previous paragraph that in this case the points $\xi^j$, $j=1,\dots, 8$, lie on the positive semi-axis and that $h_0(\xi^j)=1$, $|h_1(\xi^j)|=1$ for $j=1,\dots, 8$.  

From the definition of the function $h_0$ and from equality \eqref{for_h} for $h_0$, we derive:
\begin{equation*}
	\sum\limits_{j=1}^8 \lambda_j=\int\limits_{\Omega}h_0^2(x) d\mu=24\int\limits_{0}^{1/16}x d\mu + \int\limits_{1/16}^{1/8}(-8x+2) d\mu+\frac{7}{8}=1.
\end{equation*}
At the same time 
\begin{equation*}
	\int\limits_{0}^{1}g_{[0,1]}^2(x) d\mu=
	64\int\limits_{0}^{1/8}x^2 d\mu+\frac{1}{4}+64\int\limits_{3/8}^{5/8}\biggl(-x+\frac{1}{2}\biggr)^2d\mu+\frac{1}{4}+64\int\limits_{7/8}^{1}(x-1)^2 d\mu=\frac{2}{3}
\end{equation*}
and so
\begin{gather}\label{int_h_1}
	\int\limits_{\Omega}h_1^2(x) d\mu=\int\limits_{-1}^{0}\frac{x^2}{4} d\mu+\int\limits_{0}^{1}g_{[0,1]}^2(x) d\mu=\frac{1}{12}+\frac{2}{3}=\frac{3}{4}.
\end{gather}
Hence on the left--hand side of equality \eqref{for_h} for $h_1$ we have $3/4$ but there is $1$ on the right--hand side. We came to a contradiction. Thus, $m(X_8, w)>8$.

To prove that $X_8\in \mathcal{M}^w(9,2,0)$,  let $k=9$, $\lambda_1=-4$, $\lambda_i=1/8$, $i=2,\dots,9$, $\xi^1=-1/2$, and $\xi^2,\dots,\xi^9$ be the midpoints of the segments on which the functions $h_4$, $h_5$, $h_6$, and $h_7$ are equal to $\pm 1$ (that is,  $\xi^2=37/256$, $\xi^3=39/256$, $\xi^4=53/256$, $\xi^5=55/256$, $\xi^6=165/256$, $\xi^7=167/256$, $\xi^8=181/256$, $\xi^9=183/256$),  and check \eqref{for_h} and \eqref{for_pair}. By the construction of the functions $h_0, h_1, \dots, h_7$ and from $\lambda_i=\lambda_j$, $i,j\in\{2,\dots, 9\}$, we obtain that it is sufficient to check \eqref{for_h} for $h_0$, $h_1$, $h_2$, and $h_4$, which can be done straightforwardly. The remaining equalities are obviously true. Part a) is proved.

Turn to part b). Assume that the points $\xi^j$  and the positive weights $\lambda_j$ are such that \eqref{int_eq_k} holds with $k=9$. It follows from the proof of a) that at least eight points lie on the segment $[1/8,1]$, namely on that part of this segment, where $h_1$ equals $\pm 1$. If at the same time the ninth point belongs to $[1/8, 3/8]\cup[5/8,7/8]$, then \eqref{for_h} for $h_0$ implies that the sum of the weights $\lambda_j$, $j=1,\dots, 9$, is $1$. This contradicts equality \eqref{for_h} for $h_1$ (see the proof of a)). Thus we derive that there are four points among $\xi^j$, $j=1,\dots, 9$, on the segment $[1/8, 3/8]$ (to be definite, $\xi^2$,  $\xi^3$, $\xi^4$, $\xi^5$); two of them on the segment $[9/64, 11/64]$ and two others on  $[13/64, 15/64]$. 
Equality \eqref{for_h} for $h_2$ takes the form:
\begin{gather}\label{2_5}
 \frac{2}{3}\cdot \frac{1}{8}+5\cdot \frac{2}{3}\cdot \frac{1}{8}=\frac{1}{2}= \lambda_2+\lambda_3+\lambda_4+\lambda_5.
\end{gather}

Similarly, there are four points among $\xi^j$, $j=1,\dots, 9$, on the segment $[5/8, 7/8]$. To be definite, $\xi^6, \xi^7, \xi^8, \xi^9$. Then it follows from \eqref{for_h} for $h_3$ that
\begin{gather}\label{6_9}
\frac{1}{2}= \lambda_6+\lambda_7+\lambda_8+\lambda_9.
\end{gather}
Equalities \eqref{int_h_1}, \eqref{2_5}, and \eqref{6_9} reduce \eqref{for_h} for $h_1$ to the form:
$
3/4=1+\lambda_1 h_1^2(\xi^1).
$
So $\lambda_1<0$ and b) is proved.
\end{proof}

\begin{Remark}
One can show that $X_8\in \mathcal{M}_+^w(11,2,0)\backslash \mathcal{M}_+^w(10,2,0)$.
\end{Remark}
Actually, the following theorem holds:
\begin{Oldtheorem}[{\cite[Corollary 4.2]{DPTT}}]
	Let $\Omega$ be a sequentially compact topological space with the probability Borel measure $\mu$ and let $X_N$ be an $N$--dimensional real linear subspace of $L_1(\Omega, \mu)\cap C(\Omega)$.  Then $X_N\in \mathcal{M}^w(N(N+1)/2,2,0)$. 
\end{Oldtheorem}

For $X_N$, we define 
$
m_+(X_N, w):=\min\{m: X_N\in\mathcal{M}_+^w(m,2,0)\}$.
It is still an open question, how much the values $m(X_N, w)$ and $m_+(X_N, w)$ can differ.

{\bf Acknowledgements.} The author is grateful to V.N. Temlyakov for posing the problem at the seminar of the Laboratory ``Multidimensional Approximation and Applications'' and to K.A. Oganesyan for careful reading and useful comments.

This research was supported by the Russian Science Foundation (project no. 21-11-00131) at Lomonosov Moscow State University.


\begin{thebibliography}{99}
\bibitem{DPTT} F. Dai, A. Prymak, V.N. Temlyakov, and  S. Tikhonov, Integral norm discretization and related problems, {\it Russian Math. Surveys} {\bf 74}:4(448) (2019),   579--630.
 Translation from
{\it Uspekhi Mat. Nauk}  {\bf 74}:4(448)  (2019),	3--58; arXiv:1807.01353v1.
\end{thebibliography}
\end{document}